\documentclass[a4paper,11pt]{article}
\usepackage{hyperref}
\hypersetup{
	colorlinks=true,
	linktoc=all,
	linkcolor=black,     
	citecolor=blue,     
}
\usepackage{tikz}
\usepackage{graphicx}
\usepackage{latexsym, bm}
\usepackage{multicol}
\usepackage{indentfirst}
\usepackage{amssymb,amsfonts}
\usepackage{amsmath}
\usepackage{setspace}
\usepackage{float}
\newtheorem{theorem}{Theorem}[section]
\newtheorem{lemma}[theorem]{Lemma}

\newtheorem{definition}[theorem]{Definition}

\newtheorem{claim}[theorem]{Claim}

\newcommand{\ignore}[1]{}

\begin{document}
\begin{spacing}{1.02}
\date{}

\title{Star-critical Ramsey numbers involving large generalized books\footnote{Center for Discrete Mathematics, Fuzhou University,
Fuzhou, 350108, P.~R.~China. Email: {\tt 1075052058@qq.com, linqizhong@fzu.edu.cn, niulin2139@163.com.} Supported in part  by National Key R\&D Program of China (Grant No. 2023YFA1010202) and NSFC (No.\ 12571361).}}

\author{Peiyu Huang, \ \ \ Qizhong Lin, \ \ \ Lin Niu\footnote{Corresponding author.}}

\maketitle
\begin{abstract}
For graphs $F$, $G$, and $H$, we write $F \to (G, H)$ if every red/blue edge-coloring of $F$ contains either a red copy of $G$ or a blue copy of $H$. The Ramsey number $R(G, H)$ is the smallest integer $N$ such that $K_N \to (G, H)$. Let $H := K_k + nK_h$ be the generalized book graph, and let $G := K_{p+1}(a_1, a_2, \dots, a_{p+1})$ be a complete $(p+1)$-partite graph satisfying $a_1 = 1$, $a_2 \mid (nh-1)$, and $a_i \le a_{i+1}$.
In this paper, avoiding the use of Szemer\'{e}di's regularity lemma, we prove that for any fixed $h, p \ge 1$, $k \ge 2$, and sufficiently large $n$, $K_{p(nh + a_2k - 1) + 1} \setminus K_{1, nh - a_2(h-2) - 1}\to(G, H).$ This result yields the star-critical Ramsey number $r_*(G, H) = (p-1)(nh + a_2k - 1) + a_2(k + h - 2) + 1.$

\medskip

\end{abstract}

\section{Introduction}

For graphs $F,G$ and $H$, we write $F \to (G, H)$ if any red/blue edge coloring of $F$ contains either a red copy of $G$ or a blue copy of $H$. The Ramsey number $r(G,H)$ is defined as the smallest integer $r$ such that $K_r\to(G, H)$. By definition, there exists a \emph{critical} red/blue edge-coloring of $K_{r-1}$ that contains neither a red copy of $G$ nor a blue copy of $H$; such a coloring is called a \emph{$(G,H)$-free coloring}.

A classical result of Chv\'atal \cite{c} states that $r(K_p,T_n)= (p-1)(n-1)+1$, where $T_n$ is a tree of order $n$.
For a graph $G$, let  $\chi(G)$ be the chromatic number of $G$, and let $s(G)$ be the minimum size of a color class over all proper vertex-colorings of $G$ by $\chi(G)$ colors. Burr \cite{Burr} first observed that if $G$ is a graph with chromatic number $\chi(G)=p$  and  $H$ is a connected graph of order $n\geq s(G)$, then
\begin{align}\label{burr}
r(G,H)\ge(p-1)(n-1)+s(G).
\end{align}

In 1983, Burr and Erd\H{o}s \cite{be} initiated the study of Ramsey goodness problems. A connected graph $H$ is called {\em $G$-good} if the equality of (\ref{burr}) holds.
Thus all trees are $K_p$-good from \cite{c}.
For more Ramsey goodness results,  we refer the reader to the survey \cite[Section 2.5]{cfs-15} and the references therein.

Let $B_n^{(k)}$ denote the book graph consisting of $n$ copies of $K_{k+1}$ sharing a common $K_k$ (the \emph{base} of the book).
Book graphs have attracted considerable attention in extremal graph theory, culminating in the recent breakthrough of Campos, Griffiths, Morris and Sahasrabudhe \cite{cgms}.

Nikiforov and Rousseau \cite{N-R} proved that for fixed $p\ge 2$ and $k\ge2$, $B_n^{(k)}$ is $K_{p+1}$-good for all sufficiently large $n$, i.e.,
 $$r(K_{p+1},B_n^{(k)})=p(n+k-1)+1.$$
Recently, a breakthrough of Fox, He and Wigderson \cite{fox} showed that the same assertion already holds under the weaker condition $n\ge2^{k^{10p}}$, and their proof avoids the use of the regularity lemma \cite{sz}. Their result can be stated as follows.
\begin{theorem}[Fox, He and Wigderson \cite{fox}]\label{fox}
For every $k, p, t\geq 2$, there exists $\delta>0$ such that the following holds for all large $n$. Let $1\leq a_1\leq \dots \leq a_{p-1}\leq t$ and $a_{p}\leq\delta n$ be positive integers. If $a_1=a_2=1$, then $r(K_p(a_1,\dots,a_{p}), B^{(k)}_n)= (p-1)(n+k-1)+1$.
\end{theorem}

Building on the work of Fox, He and Wigderson \cite{fox} and strengthening a result of Lin and Liu \cite{L-Liu}, Fan and Lin \cite{Fan and Lin} established the following theorem.
\begin{theorem}[Fan and Lin \cite{Fan and Lin}]\label{Fan-Lin}
For every $k, p, t\geq 2$, there exists $\delta>0$ such that the following holds for all large $n$. Let $1\leq a_1\leq \dots \leq a_{p-1}\leq t$ and $a_{p}\leq\delta n$ be positive integers. If $a_1=1$, then $r(K_p(a_1,\dots,a_{p}), B^{(k)}_n)\le (p-1)(n+k-1)+k(p-1)(a_2-1)+1$. The equality holds if $a_2\mid (n-1)$.
\end{theorem}

A natural generalization of the book graph $B_n^{(k)}$ is the graph $K_k + nH$, obtained by replacing each vertex in the ``pages'' of the classical book by a copy of an arbitrary graph $H$. This construction preserves the essential book-like structure while allowing greater flexibility through the choice of $H$, making it a natural object for extending Ramsey-theoretic results. For the graph $K_k + nH$, it was shown in \cite{L-L-W3} that $K_k + nH$ is $(K_2 + G)$-good when $n$ is double-exponential in $k$, $|V(G)|$ and $|V(H)|$. Furthermore, by replacing $B_n^{(k)}$ with $K_k + nK_h$, Fan and Lin obtained the following generalization.
 \begin{theorem}[Fan and Lin \cite{Fan}]\label{new}
For every $k, t\geq 2$ and $h,p\ge 1$, there exists $\delta>0$ such that the following holds for all large $n$. Let $1\leq a_1\leq \dots \leq a_{p}\leq t$ and $a_{p+1}\leq\delta n$ be positive integers. If $a_1=1$ and $a_2\mid (nh-1)$, then
$r(K_{p+1}(a_1,\dots,a_{p+1}),K_k+nK_h)=p(nh+a_2 k-1)+1.$
\end{theorem}

A natural problem is whether one can remove a subgraph $F$ from $K_r$ such that $$K_r \setminus F \to (G, H)$$ still holds, where $K_r \setminus F$ denotes the graph obtained by deleting all edges of $F$ from $K_r$. In particular, when $F$ is a star, this leads to the definition of star-critical Ramsey number, which was first introduced by Hook and Isaak \cite{h-i}.

To be convenient, let $K_{r-1}\sqcup K_{1,t}$ denote the graph obtained from $K_{r-1}$ and an additional vertex $v$ by joining $v$ to $t$ vertices of $K_{r-1}$, $K_{r-1}\sqcup K_{1,t}=K_r\setminus K_{1,r-1-t}$.

 \begin{definition}
For $r=r(G,H)$, the star-critical Ramsey number is defined as
$$r_{\ast}(G,H)=\min\{t\mid K_{r-1}\sqcup K_{1,t}\to(G,H)\}.$$
\end{definition}

Recently, researches on star-critical Ramsey numbers mainly focused on classic book graphs, complete graphs, cycles, fans, and other graphs (e.g., \cite{h-l-1,h-l,L-L-W1,L-L-W2,l-c,w-l-l,w-s-r,z-b-c}).

From the above, we know that if $r(G,H)= (\chi(G)-1)(n-1)+s(G),$ then
$H$ is $G$-good.
Under all proper vertex-colorings of $G$ with $|V_1|=s(G)$ and all the color classes $V_1,\ldots,V_{\chi}$, denote $\tau(G)$ by the minimum degree of some vertex of $V_1$ that is adjacent to $V_i$ for $2\le i\le \chi$. A lower bound for the star-critical Ramsey number $r_{\ast}(G,H)$ is as follows, and a similar bound appears in \cite[Theorem 3]{z-b-c} due to Zhang, Broersma and Chen.

 \begin{lemma}[Hao and Lin \cite{h-l}]\label{lower}
Let $G$ be a graph with $\chi(G)\ge2$, and let $H$ be a connected graph of order $n\ge s(G)+1$ with minimum degree $\delta(H)$, if $H$ is $G$-good, then
\[
r_*(G,H)\ge(\chi(G)-2)(n-1)+\min\{n,\delta(H)+\tau(G)-1\}.
\]
Furthermore, if $H$ contains no vertex cut or $\delta(H)=1$, then
\[
r_*(G,H)\ge(\chi(G)-2)(n-1)+\min\{n,\delta(H)+\tau(G)-1\}+s(G)-1.
\]
\end{lemma}

In particular, Hao and Lin \cite{h-l} proved that $r_*(K_{p+1}, B_n^{(k)}) = (p-1+o(1))n$ for sufficiently large $n$, and they posed the problem of determining the exact value of $r_*(K_{p+1}, B_n^{(k)})$. Using the regularity lemma, Li and Li \cite{l-l} showed that $r_*(K_2 + G, B_n^{(k)}) = (\chi(G)+o(1))n$ for any fixed graph $G$ and large $n$. Recently, Chen, Lin and Niu \cite{chen} solved the problem of Hao and Lin by determining the exact value of $r_*(K_2 + G, B_n^{(k)})$ for any fixed $G$, and they established the following more general result.
\begin{theorem}[Chen, Lin and Niu \cite{chen}]\label{old}
For every $k, t\geq 2$ and $p\ge 1$, there exists $\delta>0$ such that the following holds for all large $n$. Let $1\leq a_1\leq \dots \leq a_{p}\leq t$ and $a_{p+1}\leq\delta n$ be positive integers. If $a_1=1$ and $a_2|(n-1)$, then
$$r_{\ast}(K_{p+1}(a_1,...,a_{p+1}),B_n^{(k)})=(p-1)(n+a_2k-1)+a_2(k-1)+1.$$
\end{theorem}

In this paper we further generalize Theorem \ref{old}. Avoiding the use of the regularity lemma, we obtain the exact value of $r_*(K_{p+1}(a_1,\dots,a_{p+1}), K_k + nK_h)$. Our main result is as follows.
\begin{theorem}\label{new star}
For every $k, t\geq 2$ and $h,p\ge 1$, there exists $\delta>0$ such that the following holds for all large $n$. Let $1\leq a_1\leq \dots \leq a_{p}\leq t$ and $a_{p+1}\leq\delta n$ be positive integers. If $a_1=1$ and $a_2\mid (nh-1)$, then
$r_{\ast}(K_{p+1}(a_1,...,a_{p+1}),K_k+nK_h)=(p-1)(nh+a_2 k-1)+a_2(k+h-2)+1.$
\end{theorem}

The bounds on $1/\delta$ obtained in our proof are not of tower-type, since we avoid the regularity lemma.

\medskip
\noindent
{\bf Notation.} Given a graph $G = (V, E)$, we write $E(G)$ for its edge set and $e(G) = |E(G)|$. For disjoint subsets $A, B \subseteq V$, let $e_G(A, B)$ denote the number of edges between $A$ and $B$ in $G$. 
For a vertex $v \in V(G)$, we denote by $N(v, G)$ its neighborhood and by $\deg(v, G)$ its degree in $G$. 
For a subset $U \subseteq V(G)$, we write $N_G(v, U)$ for the set of neighbors of $v$ in $U$, and set $\deg_G(v, U) = |N_G(v, U)|$.
We use $G \setminus U$ and $G[U]$ to denote the subgraphs induced by $V(G) \setminus U$ and by $U$, respectively. 
Finally, $K_p(a_1, \dots, a_p)$ stands for the complete $p$-partite graph whose part sizes are $a_1, \dots, a_p$.

\section{Preliminaries}

The following lemma implies that if a graph $G$ on $n$ vertices contains $\Omega(n^p)$ copies of $K_p$, then it must contain a complete \(p\)-partite subgraph $K_p(a_1,\dots,a_{p})$ with one part of size linear in $n$. For a proof, we refer the reader to Fox, He, and Wigderson \cite{fox}.
\begin{lemma}[Erd\H{o}s \cite{erdos64}]\label{Erdos}
For every $\eta>0$, and integers $p,t\geq2$, and $1\leq a_1\leq \dots \leq a_{p-1}\leq t$, there exists $\delta>0$ such that the following holds for  all sufficiently large $n$. If $G$ is a $K_p(a_1,\dots,a_{p})$-free graph on $n$ vertices with $a_p\leq \delta n$, then $G$ has at most $\eta n^p$ copies of $K_p$.
\end{lemma}

The following classical result guarantees a large independent set in any graph with low average degree.
\begin{lemma}[Tur\'{a}n \cite{Turan}]\label{Turan}
For any graph $G$ of order $n$ with average degree $d$, the independence number $\alpha$ is at least $\frac{n}{1+d}$.
\end{lemma}


We also require  the following {\em stability-supersaturation lemma} by Fox, He and Wigderson \cite[Theorem 3.1]{fox}. It is a key tool for describing the structure of graphs that forbid certain subgraphs. The lemma states that if a graph \(G\) has slightly fewer edges than the Tur\'{a}n graph for \(K_p\) and contains only a few copies of \(K_p\), then its structure must be close to that of a balanced complete \((p-1)\)-partite graph.


\begin{lemma}[Fox, He and Wigderson \cite{fox}]\label{stability}
For every $\varepsilon>0$ and every integer $p\geq 2$, there exist $\eta,\gamma>0$ such that the following holds for all $n\geq 10$. Suppose $G$ is a graph on $n$ vertices with minimum degree at least $(1-\frac{1}{p-1}-\gamma)n$ and at most $\eta n^p$ copies of $K_p$. Then there is a partition $V(G)=\sqcup_{1\le i\le p-1} V_i$, such that the following hold:

\medskip
(i) $\sum_{1\le i\le p-1}e(V_i)\le\varepsilon {n\choose 2}$.

(ii) $\big||V_i|-\frac{n}{p-1}\big|\leq\sqrt{2\varepsilon}n$.

(iii) $e_G(V_i,V_j)\geq (1-p^2\varepsilon)|V_i||V_j|$.

(iv) For each $v\in V_i$, $\deg_G(v,V_i)\leq \deg_G(v,V_j)$.

\medskip
Moreover, we may take $\gamma=\min\{\frac{1}{2p^2},\frac{\varepsilon}{2}\}$ and $\eta=p^{-10p}\varepsilon$.
\end{lemma}

Finally, recall that $G_1=K_{p+1}(a_1,\dots,a_{p+1})$ and $G_2=K_k+nK_h$. We consider a graph $G$ of order $N \ge (p-1-\eta)(n+k)$ that is both $G_1$-free and $G_2$-free. The following lemma will be essential, which shows that almost all vertices of \(G\) have high degree.

\begin{lemma}[Fan and Lin \cite{Fan and Lin}]\label{Fan and Lin}
For every $k, p, t\geq 2$, and for any $0<\alpha\le \frac{1}{2}$ and sufficiently small $\eta>0$, there exists $\delta>0$ such that the following holds for all large $n$. Let $1\leq a_1\leq \dots \leq a_{p-1}\leq t$ and $a_p\leq\delta n$ be positive integers. If $G$ is a graph on $N\ge (p-1-\eta)nh$ vertices such that $G$ is $K_p(a_1,\dots,a_{p})$-free and its complement $\overline{G}$ is $K_k+nK_h$-free, then all but at most $\alpha N$ vertices of $G$ have degree at least $(1-\frac{1}{p-1}-\alpha)N$.
\end{lemma}

\section{Proof of Theorem \ref{new}}

For completeness, we give a proof for Theorem \ref{new}.
\subsection*{3.1 Lower bound}

We shall prove $r(K_{p+1}(a_1,\dots,a_{p+1}),K_k+nK_h)\ge p(nh+a_2 k-1)+1$ where $a_1=1$. Let $F$ be the disjoint union of $\frac{nh-1}{a_2}+k$ copies of the complete graph $K_{a_2}$, and let $G_0$ be the graph obtained by taking $p$ disjoint copies $F_1,\dots,F_p$ of $F$ and joining every vertex of $F_i$ to every vertex of $F_j$ for all $i \neq j$.  

We first show that $G_0$ contains no $\mathcal{K}:=K_{p+1}\left(a_1,\dots,a_{p+1}\right)$ by induction on $p\ge1$. To be clear, let $V_i$, $1\le i\le p$, be the vertex set of $F_i$.  The base case when $p=1$ is clear, and we assume that there is no $\mathcal{K}$ for smaller $p$. Now suppose that, on the contrary, $G_0$ contains $\mathcal{K}$ as a subgraph, where $\{u_0\}$ corresponds to the class consists of the single vertex. By symmetry, we assume $u_0\in V_{p}$. Since $u_0$ is adjacent to all vertices of $V(\mathcal{K} )\setminus(\{u_0\}\sqcup N_{V_p}(u_0))$, the subgraph $\mathcal{K}\setminus(\{u_0\}\sqcup N_{V_p}(u_0))$ must be contained in $G_0\setminus V_p$. Note that $u_0$ has only $a_2-1$ neighbors in $V_p$ in $G_0$, we have $K_{p}(1,a_3,\dots,a_{p+1})\subseteq G_0\setminus V_{p}$, which is impossible from the induction hypothesis.

Now, let us check that $\overline{G_0}$ contains no $K_k+nK_h$. If there is a $K_k+nK_h$, then such a subgraph must be in $\overline{F_i}$ for some $1\le i\le p$, which is impossible since the number of the common neighbors of any clique of $k$ vertices in $\overline{F_i}$ is at most $a_2(\frac{nh-1}{a_2}+k-k)=nh-1$. Thus, the lower bound $r(K_{p+1}(a_1,\dots,a_{p+1}),K_k+nK_h)\ge p(nh+a_2 k-1)+1$ holds as desired.

\subsection*{3.2 Upper bound}

 We now turn to the upper bound. Let $N=p(nh+a_2k-1)+1$. For the complete graph $K_N$, denote by $R$ and $B$ the graphs induced by all red edges and blue edges respectively. Consider a $(\mathcal{K},K_k+nK_h)$-free coloring of $K_N$, i.e., the red graph $R$ contains no copy of $K_{p+1}(a_1,\dots,a_{p+1})$ and the blue graph $B$ contains no copy of $K_k+nK_h$. Lemma \ref{Fan and Lin} shows that at most $\gamma N$ vertices of $K_N$ have red degree at most $\lceil(1-\frac{1}{p}-\gamma)N\rceil.$
Let $T$ be the set of vertices of red degree greater than $\lceil(1-\frac{1}{p}-\gamma)N\rceil$, then the induced subgraph $R[T]$ has at least $(1-\gamma)N$ vertices, and thus its minimum degree is at least $(1-\frac{1}{p}-2\gamma)|T|$. 
 Lemma \ref{Erdos} and Lemma \ref{stability} imply that if $n$ is sufficiently large, then there exists a partition of $T$ into classes $T_1,T_2,\dots,T_p$ such that

\medskip
(i) $\sum_{1\le i\le p}e_R(T_i)\le\varepsilon {(1-\gamma)N\choose 2}$.

(ii) $\big||T_i|-\frac{(1-\gamma)N}{p}\big|\leq\sqrt{2\varepsilon}(1-\gamma)N$.

(iii) $e_R(T_i,T_j)\geq\left(1-(p+1)^2\varepsilon\right)|T_i||T_j|$.

(iv) For each $v\in T_i$, $\deg_R(v,T_i)\leq\deg_R(v,T_j)$.

\medskip
Moreover, we may take $\gamma=\min\{\frac{1}{4(p+1)^2},\frac{\varepsilon}{4}\}$ and $\eta=(p+1)^{-10(p+1)}\varepsilon$.

\medskip
Let $S=V(K_N)\setminus T$, where $|S|\leq \gamma N$. Adapting the argument from \cite[Theorem 1.4]{Fan and Lin} yields the following result.

\begin{claim}\label{ext-fhw}
Let $\zeta=4(p+1)^3t\sqrt\varepsilon$. For every $1\leq i\leq p$, we have the following properties.

\medskip

(i) For every $v\in T_i$, $\deg_R(v,T_i)\le a_2-1$.

(ii) For every $v\in S$, either $\deg_R(v,T_i)\le a_2-1$ for some $T_i$, or else there exist distinct $k,\ell$ such that $\deg_R(v,T_i)<\zeta|T_i|$ for $i=k,\ell$.
\end{claim}

Based on Claim \ref{ext-fhw}, we first partition $S$ into subsets $S_1$ and $S_2$, where $S_1$ consists of vertices $v \in S$ such that $\deg_R(v, T_i)\le a_2-1$ for some $i$, and $S_2$ consists of the remaining vertices of $S$, namely those satisfying $\deg_R(v,T_i)<\zeta|T_i|$ for at least two choices of $1\leq i\leq p$. Moreover, \(S_1\) is divided into \(S_{11}, \dots, S_{1p}\) and \(S_2\) into \(S_{21}, \dots, S_{2p}\), both according to the smallest index \(i\) satisfying the respective degree condition. We then define $V=V(K_N)$ and $V_i=T_i\cup S_{1i}\cup S_{2i}$. Consequently, we obtain $V=\bigcup^{p}_{i=1}V_i$.

By a similar argument as in \cite[Theorem 1.3]{chung}, we obtain the following claim.
\begin{claim}\label{r}
Let $\delta>0$, $h$ and $t$ be fixed positive integers, and let $1= a_1\leq a_2\leq\dots\leq a_p\leq t$ and $a_{p+1}\leq\delta n$ be positive integers. Then $r(K_h,\mathcal{K})\le (a+1)^h\delta n$, where $a=\sum_{i=1}^{p}a_i$.
\end{claim}

Now, we write $\ell$ for $r(K_h,\mathcal{K})$ and consider two cases based on the size of $S_{2}$.

\medskip
{\em Case 1.} $|S_{2}|\ge2(p+1)$.

\medskip
 From Claim \ref{ext-fhw}(ii), we know that $S_2=\bigcup_{i=1}^{p} S_{2i}$. Now we partition $S_2$ in another way. Indeed, we partition $S_2$ into $p$ subsets $S_{21}',\dots,S_{2p}'$ such that a vertex $v\in S_{2}$ is put into $S_{2i}'$ for $1\leq i\leq p$ if and only if $i$ is the biggest index satisfying $\deg_R(v,T_i)<\zeta|T_i|$.
Note that the definitions of $S_{2i}$ and $S_{2i}'$ imply that $S_{2i}\cap S_{2i}'=\emptyset$, and hence $|S_{2i}|+|S_{2i}'|\le |S_2|$. Define
$V_i=T_i\cup S_{1i}\cup S_{2i}\cup S_{2i}'$. By the pigeon-hole principle, there exists some $h\in [p]$ such that
$$|V_h|\ge\frac{|V|+|S_2|}{p}\ge nh+a_2 k-1+\frac{|S_2|}{p}.$$

From Claim \ref{ext-fhw}(i) and Lemma \ref{Turan}, there exists a blue clique $B[X]$ with $X\subseteq T_h$ and $|X|\ge \frac{|T_h|}{a_2}$. We randomly and uniformly select a $k$-set from $X$, which is a blue clique.
By linearity of expectation, there exists a $k$-set $K \subseteq X$ whose common blue neighborhood $B_0$ in $T_h \cup S_{1h} \cup S_{2h} \cup S'_{2h}$ satisfies:
$$|B_0|\ge (|T_h|-ka_2)+(|S_{2h}|+|S_{2h}'|)(1-ka_2\zeta)+|S_{1h}|\left(1-k\frac{a_2-1}{|X|}\right)\ge nh,$$
where the last inequality follows since $|S_{2}|\ge 2(p+1)$ and $\zeta, \gamma$ are sufficiently small.

\begin{claim}\label{claim nh}
 $B_0$ contains $n$ disjoint blue cliques of size $h$.
 \end{claim}
{\bf Proof.} We may assume $|B_0 \cap (S_{2h} \cup S'_{2h})| \geq \ell$. Write $|B_0 \cap (S_{2h} \cup S'_{2h})| = \ell + n_1h + h'_1$, where $0 \leq h'_1 < h$. Since $B_0 \cap (S_{2h} \cup S'_{2h})$ contains no red $\mathcal{K}$, we can find at least $n_1 + 1$ disjoint blue $K_h$ within this set. Let $Y_0$ denote the remaining vertices in $B_0\cap (S_{2h}\cup S_{2h}')$  after deleting the vertices of these disjoint blue  $h$-cliques. Then $|Y_0|<\ell<(a+1)^h\delta n$ from Claim \ref{r}. Note that each vertex $v\in K$ has at most $a_2-1$ red neighbors in $T_{h}$, and $v\in Y_0$ has less than $ \zeta |T_h|$  red neighbors in $T_{h}$ by Claim \ref{ext-fhw}. Obviously, $\deg_B(v, T_h) > (1-\zeta)|T_h|$ for any $v \in Y_0$. Moreover, $(1-\zeta)|T_h| - ka_2 \geq (a_2(h-2)+1)\ell$, where the last inequality follows since $\delta$, $\zeta$, and $\gamma$ are sufficiently small. Noticing that $a_2(h-2)+1=r(K_{1,a_2},K_{h-1})$ and $R[T_{h}]$ contains no red $K_{1,a_2}$, any vertex $v\in Y_0$ and $h-1$ vertices of $B_0\cap(T_h\setminus K)$ form a blue clique of size $h$.  

Let $Y_1$ consist of the vertices in $B_0 \cap S_{1h}$ together with those remaining in $B_0 \cap (T_h \setminus K)$ after deleting the vertices of the disjoint blue $(h-1)$-cliques. Thus, $|Y_1|>|T_h|-ka_2-\ell(h-1)$.
Let $B_1$ denote a maximum blue clique in the remaining vertices in $B_0\cap (T_{h}\setminus K)$. Applying Lemma \ref{Turan}, we obtain that $|B_1|\ge(|T_h|-ka_2-(h-1)\ell)/a_2$.

 We firstly assume that \( |Y_1\setminus B_1| < \ell \). Since each vertex \( v \in Y_1 \setminus B_1 \) has at most \( a_2 - 1 \) red neighbors in \( B_1 \) and \( |B_1| >|Y_1|-\ell >|T_h|-ka_2-h\ell\geq (a_2 + h - 2)\ell \), we obtain that the vertices of $Y_1\setminus B_1$ together with their blue neighbors in \( B_1\) form $|Y_1\setminus B_1|$  blue cliques. Clearly, the remaining vertices in $B_1$ form a certain number of blue $h$-cliques. Therefore, \( B_0\) contains \( n \) disjoint blue cliques, each of size \( h \).

Now, we assume that $|Y_1\setminus B_1|\ge\ell $, and let $|Y_1\setminus B_1|=\ell+n_2h+h_2'$, where $0\le
h_2'<h$. Since there is no red $\mathcal{K}$, we then can find at least
$n_2+1$ disjoint blue $h$-cliques in $Y_1\setminus B_1$. Let $X_0$ denote the remaining vertices in $Y_1\setminus B_1$  after deleting the vertices of these disjoint blue $h$-cliques. Then $|X_0|<\ell<(a+1)^h\delta n$. From Claim \ref{ext-fhw}(i), each vertex $v\in X_0$ has at most $a_2-1$ red neighbors in $B_1$. It follows that there is a subset $B_2$ of $B_1$ consisting of vertices blue-adjacent to any vertex in $X_0$ satisfying
$$
|B_2|\ge |B_1|-(a_2-1)\ell\ge (|T_h|-ka_2-(h-1)\ell)/a_2-(a_2-1)\ell\ge(h-1)\ell
$$
where the last inequality follows from the fact that $ \delta$, $\zeta$, and$\gamma$ are sufficiently small. Thus any vertex $v\in X_0$ and $h-1$ vertices of $B_2$ form a
blue clique of size $h$. Let $B_3\subset B_2$ be the set  consisting of the vertices that have been accounted for.  The remaining vertices in $B_1\setminus B_3$ clearly form blue cliques. Since $|B_0|\ge hn$, we can obtain $n$ disjoint blue cliques of size $h$  as desired.  \hfill$\Box$

The above discussion implies that $K\cup B_0$ yields a $K_k+nK_h$ in $B$ with center $K_k$. This leads to a contradiction.

\medskip
{\em Case 2.} $|S_{2}|<2(p+1)$.

\medskip 

 Let $V_i$ be the largest part among $V_1,\dots,V_{p}$. Given $N=p(nh+a_2k-1)+1$, we have $|V_i|\ge nh+a_2 k$.
From Claim \ref{ext-fhw} (i) and Lemma \ref{Turan}, there exists a blue clique $B[X]$ with $X\subseteq T_i$ and $|X|\ge \frac{|T_i|}{a_2}$. By the definition of $S_{1i}$ and $S_{2i}$, we can find a blue clique $B[X']\subseteq B[X]$ such that any vertex of $X'$ is blue-adjacent to all vertices of $S_{1i}\cup S_{2i}$ with
$$|X'|\ge |X|-(a_2-1)|S_{1i}|-\zeta|T_i||S_{2i}|.$$

Since $(\frac{1}{p}-\gamma)N\le|V_i|-|S|\le|T_i|\le\frac{(1-\gamma)N}{p}+\sqrt{2\varepsilon}(1-\gamma)N$ and $|S_{2i}|<2(p+1)$, we obtain that $|X'|\ge k$ by noting that $\varepsilon$, $\gamma$ and $\zeta$ are sufficiently small.

Applying Claim \ref{ext-fhw} (i) again, we can find a blue \( k \)-clique \( K' \) in \( B[X'] \) such that \( K' \) is blue-adjacent to all but at most \( k(a_2 - 1) \) vertices of \( T_i \). Hence $K'$ has at least $|V_i|-k-k(a_2-1)\ge nh$ common blue neighbors which are denoted by $B_0'$. Let $B_1'$ denote the vertices in the blue clique $B[X\setminus K']$, then $|B_1'|\ge|T_i|/a_2-k$. Following an argument similar to that in Claim \ref{claim nh}, we show that $B_0'$ contains $n$ disjoint blue $h$-cliques, regardless of whether $|B_0' \setminus B_1'| \ge \ell$ or not. It follows that $K'\cup B_0'$ yields a $K_k+nK_h$ in $B$ with center $K_k$. The final contradiction completes the proof of Theorem \ref{new}.

\section{Proof of Theorem \ref{new star}}
It is known that for every $k, t\geq 2$ and $h, p\ge 1$, there exists $\delta>0$ such that the following holds for all large $n$. Let $1\leq a_1\leq \dots \leq a_{p}\leq t$ and $a_{p+1}\leq\delta n$ be positive integers. If $a_1=1$ and $a_2\mid (nh-1)$, then $r:=r(K_{p+1}\left(a_1,\dots,a_{p+1}\right), K_k+nK_h)= p(nh+a_2k-1)+1$. In the following, we will prove Theorem \ref{new star}.

\subsection*{4.1 Lower bound}

For the lower bound, we shall show a $(K_{p+1}\left(a_1,\dots,a_{p+1}\right), K_k+nK_h)$-free coloring of $G=K_{r-1}\sqcup K_{1,q}$, where $q=(p-1)(nh+a_2 k-1)+a_2(k+h-2).$ Define $F$ as $\frac{nh-1}{a_2}+k$ disjoint copies of the complete graph $K_{a_2}$, and let $G_0$ be the graph obtained from $p$ disjoint copies $F_1,F_2,\dots,F_p$ of $F$ by adding all edges between them. Moreover, let $H_1$ be the graph obtained from $G_0$ and a new vertex $u$ by adding all edges between $u$ and $\cup_{i=1}^{p-1} F_i$. Let $H_2$ be the graph obtained from $\overline{G}_0$ by adding all edges between $u$ and $k+h-2$ disjoint copies of $\overline{K}_{a_2}$ of $\overline{F}_p$. Obviously, $E(H_1)\cap E(H_2)=\emptyset$ and $H_1\cup H_2=G=K_{r-1}\sqcup K_{1,q}$.

We first show that $H_1$ contains no $\mathcal{K}:=K_{p+1}\left(a_1,\dots,a_{p+1}\right)$ (one can also see in \cite{Fan and Lin,L-Liu}). Clearly, such a $\mathcal{K}$ can not contain $u$. We now prove that $G_0=H_1\setminus\{u\}$ contains no $\mathcal{K}$ by induction on $p\ge1$. To be clear, let $V_i$, $1\le i\le p$, be the vertex set of $F_i$.  The base case when $p=1$ is clear, and we assume that there is no $\mathcal{K}$ for smaller $p$. Now suppose that, on the contrary, $G_0$ contains $\mathcal{K}$ as a subgraph, where $\{u_0\}$ corresponds to the class consisting of the single vertex. By symmetry, we assume $u_0\in V_{p}$. Since $u_0$ is adjacent to all vertices of $V(\mathcal{K})\setminus(\{u_0\}\sqcup N_{V_p}(u_0))$, the subgraph $\mathcal{K}\setminus(\{u_0\}\sqcup N_{V_p}(u_0))$ must be contained in $G_0\setminus V_p$. Note that $u_0$ has only $a_2-1$ neighbors in $V_p$ in $G_0$, then we have $K_{p}(1,a_3,\dots,a_{p+1})\subseteq G_0\setminus V_{p}$, which is impossible according to the induction hypothesis.

Now, let us check that $H_2$ contains no $K_k+nK_h$. If there is a $K_k+nK_h$, then such a subgraph must be in $\overline{F_i}$ for some $1\le i\le p$, which is impossible since the number of the common neighbors of any clique of $k$ vertices in $\overline{F_i}$ is at most $a_2(\frac{nh-1}{a_2}+k-k)=nh-1$. Thus, the lower bound $ r_{\ast}(K_{p+1}(a_1,\dots,a_{p+1}), K_k+nK_h)\ge q+1$ follows as desired.
\begin{figure}[H]
    \centering
    \includegraphics[width=0.9\textwidth]{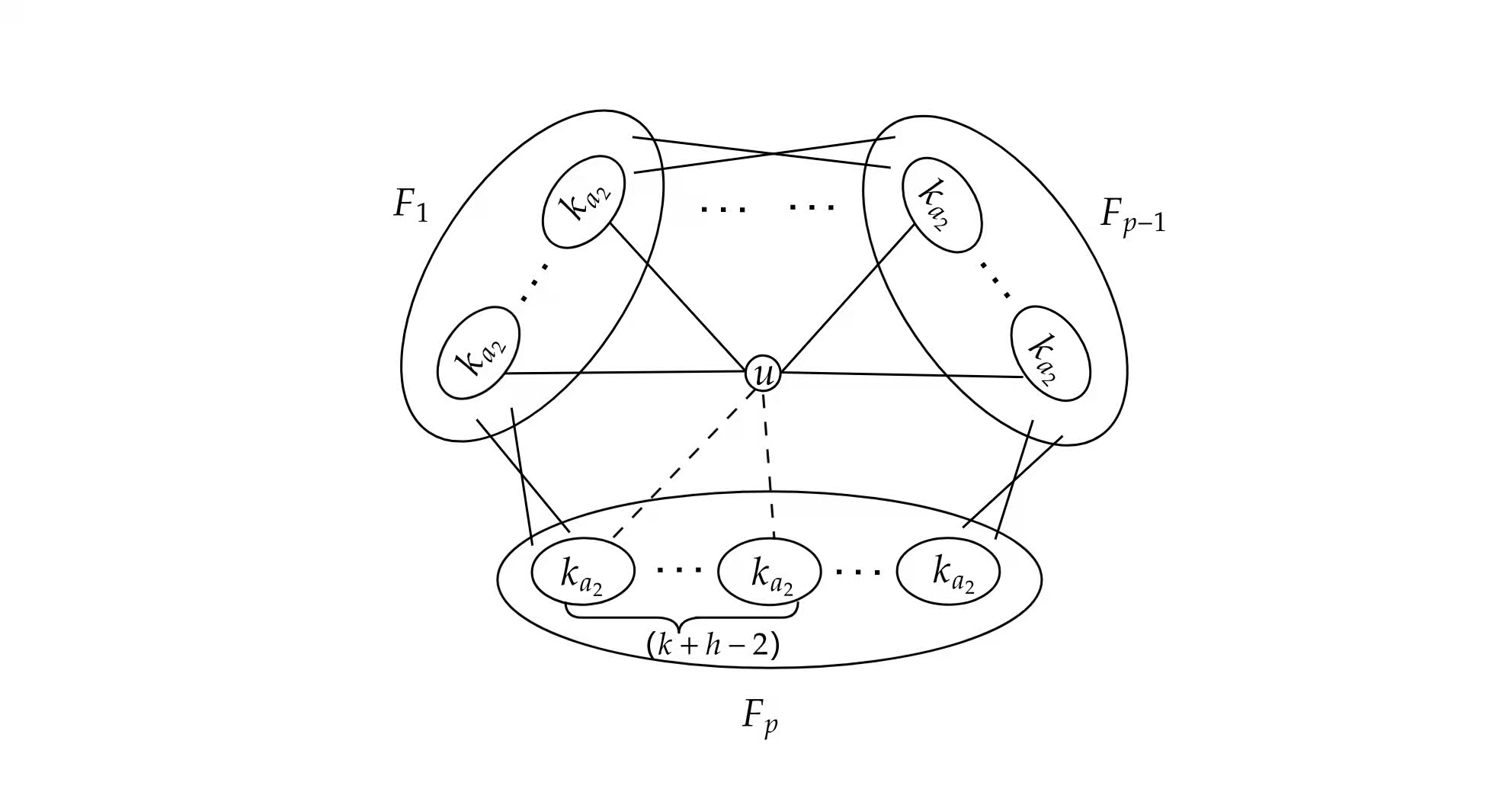}
    \caption{Construction of the graph $G$}
    \label{fig:1.}
\end{figure}

\subsection*{4.2 Upper bound}
 We now turn to the upper bound. Let $N=p(nh+a_2k-1)$. For the complete graph $K_N$, denote by $R$ and $B$ the graphs induced by all red edges and blue edges respectively. Consider a $(\mathcal{K},K_k+nK_h)$-free coloring of $K_N$, i.e., the red graph $R$ contains no copy of $K_{p+1}(a_1,...,a_{p+1})$ and the blue graph $B$ contains no copy of $K_k+nK_h$. Let $T$ be the set of vertices of red degree greater than $\lceil(1-\frac{1}{p}-\gamma)N\rceil$, and $S=V(K_N)\setminus T$. As in the proof of Theorem \ref{new}, there exists a partition of $T$ into classes $T_1,T_2,\dots,T_p$ for sufficiently large n. The same argument yields the following claim.
\begin{claim}\label{ext-fhw1}
Let $\zeta=4(p+1)^3t\sqrt\varepsilon$. For every $1\leq i\leq p$, we have the following properties.

\medskip

(i) For every $v\in T_i$, $\deg_R(v,T_i)\le a_2-1$.

(ii) For every $v\in S$, either $\deg_R(v,T_i)\le a_2-1$ for some $T_i$, or else there exist distinct $k,\ell$ such that $\deg_R(v,T_i)<\zeta|T_i|$ for $i=k,\ell$.
\end{claim}

Based on Claim \ref{ext-fhw1}, we follow the same partition method described in Theorem \ref{new}. Then we have $S = S_1 \cup S_2$, where \(S_1\) consists of vertices with \(\deg_R(v, T_i) \le a_2 - 1\) for some \(i\), and \(S_2\) consists of vertices with \(\deg_R(v, T_i) < \zeta |T_i|\) for at least two indices \(i\). Moreover, \(S_1\) is divided into \(S_{11}, \dots, S_{1p}\) and \(S_2\) into \(S_{21}, \dots, S_{2p}\), both according to the smallest index \(i\) satisfying the respective degree condition. Finally, we set $V_i = T_i \cup S_{1i} \cup S_{2i}$, then \(V = \bigcup_{i=1}^p V_i\).

With the same partition method as that of Theorem \ref{new}, we have that $|S_{2}|<2(p+1)$ and $|V_i|=nh+a_2k-1$ by noting $N=p(nh+a_2k-1)$. 

\begin{claim}\label{claim45-2}
$|S_{2i}|=0$.
 \end{claim}

\noindent{\bf Proof.}  Suppose to the contrary that $|S_{2i}|\ge 1$. By the definition of $S_{2i}$, any vertex $u\in S_{2i}$ satisfies
$$\deg_R(u,T_j)< \zeta |T_j|$$
for some $j$. Following an argument similar to that in the second case of Theorem \ref{new}, we can find a blue k-clique which has at least $nh$ common blue neighbors in $V_j\cup\{u\}$. Then, using an argument similar to Claim \ref{claim nh}, we can find a blue $K_k+nK_h$, which leads to a contradiction.\hfill$\Box$

\begin{claim}\label{claim-deg-R}
For every $v\in V_i$, $\deg_B(v,T_j)\le (a_2-1)\gamma N+a_2(k-1)$ for any $j\ne i$.
 \end{claim}

 \noindent{\bf Proof.}  Suppose to the contrary that $\deg_B(v,T_j)\ge (a_2-1)\gamma N+a_2(k-1)+1$ for some $v\in V_i$ and $j\ne i$.
 By the definition of $S_{1j}$, there exists at least $a_2(k-1)+1$ vertices of $T_j$ which are adjacent to all vertices of $S_{1j}\cup\{v\}$ with blue edges. By Claim \ref{ext-fhw1} (i), $R[T_j]$ contains no red $K_{1,a_2}$. Noting that $a_2(k-1)+1=r(K_{1,a_2},K_k)$, there exists a blue $K_k$ which are adjacent to all vertices of $S_{1j}\cup\{v\}$ with blue edges. Now, applying Claim \ref{ext-fhw1} (i) again, these $k$ vertices have at least
$$|V_j|+1-k-k(a_2-1)\ge nh$$
common blue neighbors. Using an argument similar to Claim \ref{claim nh}, we can find a blue $K_k+nK_h$ in $V_j\cup\{v\}$. A contradiction.
\hfill$\Box$

\begin{claim}\label{deg-v-U}
For every $v\in V_i$, $i=1, 2, \dots, p$, $\deg_R(v,V_i)=a_2-1$.
\end{claim}
  \noindent{\bf Proof.} We first prove that $\deg_R(v,V_i)\le a_2-1$.
By symmetry, we will prove the assertion for $i=1$. Suppose to the contrary, there exists a red $K_{1,a_2}$ in $R[V_1]$. By Claim \ref{claim-deg-R}, we have that $\deg_B(v,T_j)\le (a_2-1)\gamma N+a_2(k-1)$ for each $v\in V_1$ and $j\ne 1$. So we obtain
$$
\left|\bigcap_{v\in V(K_{1,a_2})}N_R(v,T_2)\right|\ge|T_2|-(a_2+1)[(a_2-1)\gamma N+a_2(k-1)]\ge a_3.
$$
Thus we can choose $a_3$ vertices from $\bigcap_{v\in V(K_{1,a_2})}N_R(v,T_2)$, which together with the red $K_{1,a_2}$ in $R[V_1]$ yield a red $K_3(1,a_2,a_3)$.
Inductively, we can get a red $\mathcal{K}$, which is a contradiction. Thus, $\deg_R(v,V_i)\le a_2-1$.

Next, we will prove that $\deg_R(v,V_i)=a_2-1$. Suppose to the contrary that there exists some vertex $v\in V_i$ such that $\deg_R(v,V_i)\le a_2-2$, and so $\deg_B(v,V_i)\ge nh+a_2(k-1)$.
Obviously, we can find a blue $k$-clique which contains $v$, and these $k$ vertices have at least
$$|V_i|-k-(a_2-2)-(k-1)(a_2-1)\ge nh$$
common blue neighbors. Similarly, using an argument similar to Claim \ref{claim nh}, we can obtain a blue $K_k+nK_h$. A contradiction.
 \hfill$\Box$

 \begin{claim}\label{claim-deg-R-2}
For every $v\in V_i$, $\deg_B(v,V_j)\le a_2(k+h-2)$ for any $j\ne i$.
 \end{claim}

 \noindent{\bf Proof.} Suppose to the contrary that $\deg_B(v,V_j)\ge a_2(k+h-2)+1$ for some $v\in V_i$ and $j\ne i$.
By Claim \ref{ext-fhw1}(i),  $R[V_j]$ contains no red $K_{1,a_2}$. Noticing that $a_2(k+h-2)+1=r(K_{1,a_2},K_{k+h-1})$, there exists a blue $K_{k+h-1}$ which is adjacent to $v$ with blue edges. 

Now, choosing a blue $k$-clique from $V(K_{k+h-1})$ uniformly at random, which is denoted by $K''$. Denote $U_j=V(K_{k+h-1})\setminus K''$, and let $X$ be the blue neighborhood of $K''$ in $V_j$. Then, $|X|\ge hn-1$ and $|X\setminus U_j|\ge h(n-1)$. By Claim \ref{deg-v-U}, every vertex $x\in X\setminus U_j$ has $a_2-1$ neighbors in $X$. Therefore, Lemma \ref{Turan} implies that there is a blue clique $\mathcal{B}$ with size at least $ |X\setminus U_j|/a_2 \ge h(n-1)/a_2$ in $X$. Following an argument similar to that in Claim \ref{claim nh}, we show that $X\setminus U_j$ contains $n-1$ disjoint blue $h$-cliques, regardless of whether $|(X\setminus U_j) \setminus\mathcal{B}| \ge \ell$ or not. Therefore, vertex $v$ and $h-1$ vertices of  $U_j$ form a blue $K_h$, which together with $K'' \cup (X\setminus U_j)$ yields a blue $K_k+nK_h$. This is impossible.\hfill$\Box$

\medskip
Therefore, from Claim $\ref{deg-v-U}$ and Claim $\ref{claim-deg-R-2}$, each $V_i$ induces a blue $(nh+a_2(k-1)-1)$-regular graph with $|V_i|=nh+a_2k-1$ and $\deg_B(u,V_j) \le a_2(k+h-2)$ for each $u\in V_i$, where $1\le i,j\le p$.
Let $v_0$ be a new vertex outside $V(K_N)$, we claim that if $v_0$ is adjacent to
 $$N'=(p-1)(nh+a_2 k-1)+a_2(k+h-2)+1$$
 vertices in $V(K_N)$, then there exists either a red $K_{p+1}(1,a_2,\dots,a_{p+1})$ or a blue $K_k+nK_h$.
 To avoid a blue $K_k+nK_h$, we must have $\deg_B(v_0, V_i)\le a_2(k+h-2)$. Otherwise, by a similar argument as Claim $\ref{claim-deg-R-2}$, we can find a blue $K_k+nK_h$ in $V_i\cup\{v_0\}$, which leads to a contradiction. Thus $\deg_R(v_0, V_i)\ge 1$. Without loss of generality, assume $\deg_R(v_0, V_1)\le \deg_R(v_0, V_2)\le\cdots\le \deg_R(v_0, V_p)$. Then for $1\le i\le p$, we have
$$\deg_R(v_0, V_i)\ge \frac{(i-1)(nh+2a_2-ha_2-1)+1}{i}.$$

Recall from Claims $\ref{deg-v-U}$ that every vertex $v \in V_i$ satisfies $\deg_R(v, V_i) = a_2 - 1$. Therefore, we can choose one vertex $u_1\in N_R(v_0, V_1)$, yielding a red $K_{1,a_2}$ with center $u_1$.
Note that
$$
\left|N_R(v_0, V_2)\cap N_R(u_1, V_2)\cap \bigcap_{u\in N_R(u_1,V_1)}N_R(u, V_2)\right|\ge \frac{n+2a_2-ha_2}{2}-a_2^2(k+h-2)>a_3.
$$
Hence we can select $a_3$ vertices from this intersection, obtaining a red $K_3(1, a_2, a_3)$. 
Proceeding inductively, we eventually get a red $\mathcal{K}$ , which leads to a contradiction.

 This completes the proof of Theorem \ref{new star}.
  \hfill$\Box$

\ignore{
\section{Acknowledgement}
  This research is partially supported by National Key Research and Development Program of China (Grant No. 2023YFA1010202), NFSC(Grant No. 12571361).
}

\end{spacing}
\end{document}